\documentclass[12pt]{amsart}
\usepackage{amsmath, amssymb, mathtools}
\usepackage[utf8]{inputenc}
\usepackage[T1]{fontenc}
\usepackage[style=numeric, backend=biber]{biblatex}
\usepackage{hyperref}
\hypersetup{
    colorlinks=true,
    linkcolor=blue,
    filecolor=magenta,      
    urlcolor=cyan,
    citecolor=red,
}

\usepackage{geometry}
\usepackage{setspace}
\usepackage{xcolor}
\usepackage[pagewise]{lineno}
\mathtoolsset{showonlyrefs=true}
\makeatletter
\@namedef{subjclassname@2020}{\textup{2020} Mathematics Subject Classification}
\makeatother
\numberwithin{equation}{section}
\theoremstyle{plain}
\newtheorem{thm}{Theorem}
\newtheorem{thma}{Theorem}[section]
\newtheorem{cor}{Corollary}[section]
\newtheorem{lem}{Lemma}[section]

\newtheorem{prop}{Proposition}[section]
\theoremstyle{remark}
\newtheorem{rmk}{Remark}[section]
\newcommand{\mrm}{\mathrm}

\newcommand{\mbb}{\mathbb}

\newcommand{\q}{\quad}
\newcommand{\md}{\text{ mod }}

\author{Pramath Anamby}
\address{Department of Mathematics\\
Harish-Chandra Research Institute\\
Prayagraj - 211019, India.}
\email{pramathav@hri.res.in, pramath.anamby@gmail.com}
\title[Non-vanishing of theta components]{Non-vanishing of theta components of Jacobi forms with level and an application}
\begin{document}
\begin{abstract}
We prove that a non--zero Jacobi form of arbitrary level $N$ and square--free index $m_1m_2$ with $m_1|N$ and $(N,m_2)=1$ has a non--zero theta component $h_\mu$ with either $(\mu,2m_1m_2)=1$ or $(\mu,2m_1m_2)\nmid 2m_2$. As an application, we prove that a non--zero Siegel cusp form $F$ of degree $2$ and an odd level $N$ in the Atkin--Lehner type newspace is determined by fundamental Fourier coefficients up to a divisor of $N$.
\end{abstract}
\maketitle
\section{Introduction}
The invariance of Jacobi forms with respect to $\mbb Z^2$ gives us the  decomposition of a Jacobi form into \textit{theta components}, which is a collection of one variable holomorphic functions with automorphic properties under $\mrm{SL}(2,\mbb Z)$. These theta components act as a bridge between different type of modular forms. For example, through Fourier--Jacobi coefficients and the Eichler--Zagier map, they give an important map between the Siegel modular forms of degree $2$ and the half integral weight modular forms. The theta decomposition also gives us the mappings from the space of Jacobi forms to the spaces of vector--valued modular forms and the half--integral weight modular forms. Moreover, the knowledge about the theta components is equivalent to knowing the Jacobi form itself. Thus it is important to understand the properties of these theta components.

One such interesting problem is the non--vanishing of theta components of a Jacobi form. In \cite{skogman2004fourier}, a non--vanishing result was proved for the theta components of Jacobi forms of full level and index $p$, $p^2$ or $pq$, where $p$ and $q$ are primes. In particular, for such indices, by analysing the structure of square classes, it was shown that the Jacobi form is uniquely determined by one of the associated theta components. The study of these objects has many important arithmetic consequences. For example, in \cite{bocherer2018fundamental}, non--vanishing of theta components of Jacobi forms of full level and matrix index M with discriminant of M square-free, proved to be a crucial step in studying the fundamental Fourier coefficients of Siegel modular forms, which further had applications in establishing the functional equation of the spinor zeta function of degree 3. 

In this article, we study the non--vanishing of theta components of Jacobi forms for the congruence subgroup $\Gamma_0(N)$. For Jacobi cusp forms of even weight, square--free level $N$ and index $p$, a prime with $(p,2N)=1$, it was proved by Martin in \cite{martin2020degree} that all the theta components are non--zero. Using the techniques from \cite{bocherer2018fundamental} and the properties of generalized quadratic Gauss sums, we prove the following non--vanishing result for the theta components of a Jacobi form of general level $N$ and square--free index.
\begin{thm}\label{th:thetanz}
Let $N$ be an odd integer and $m_1$, $m_2$ be square--free integers such that $m_1|N$ and $(N,m_2)=1$. For any non--zero $\phi\in J_{k,m_1m_2}(N)$ there exists a $\mu\md 2m_1m_2$ with either $(\mu,2m_1m_2)=1$ or $(\mu,2m_1m_2)\nmid 2m_2$ such that the theta component $h_\mu\neq 0$.
\end{thm}
Theorem \ref{th:thetanz} is a generalization of the non--vanishing result of \cite{bocherer2018fundamental} for Jacobi forms with level in degree 1. It will be interesting to investigate this non--vanishing question for non square--free indices and also for higher degree Jacobi forms.

As an application, we prove the following result for the Siegel cusp forms of degree $2$ in the Atkin--Lehner type newspace (see section \ref{sec:newforms}). For positive integers $k$ and $N$, let $S_k^2(N)$ denote the space of Siegel cusp forms of degree $2$, weight $k$ and level $N$. $S_k^{2,new}(N)\subseteq S_k^2(N)$ be the subspace of Atkin--Lehner type newforms (see section \ref{sec:newforms}). The Fourier coefficients $A(F,T)$ of any $F\in S_k^2(N)$ are supported on $\Lambda_2^+$, the set of $2\times 2$ half--integral, symmetric, positive definite matrices. 
 
\begin{thm}\label{thm:fund}
Let $k>2$ and $N=p_1^{\alpha_1}p_2^{\alpha_2}...p_t^{\alpha_t}$ be an odd integer. Suppose $F\in S_k^{2,new}(N)$ is non--zero, then there exist infinitely many $\mrm{GL}(2,\mbb Z)$--inequivalent $T\in \Lambda_2^+$ such that 
\begin{enumerate}
\item $A(F, T)\neq 0$.
\item $4 \mrm{det}(T)$ is of the form $\prod_{1}^{t}p_j^{r_j} n$ where $n$ is odd and square--free, $(n, N)=1$ and $0\le r_j\le \alpha_j$.
\end{enumerate}
\end{thm}
In particular, when $N$ is an odd, square--free integer we see that $F$ is in fact determined by its fundamental Fourier coefficients. When $k$ is even, this was known from \cite{saha2012determination} and Theorem \ref{thm:fund} provides an alternative proof of this fact.

The proof of Theorem \ref{thm:fund} uses the quite useful technique of reducing the question to half--integral weight cusp forms by using the Jacobi forms as an intermediate  bridge. The passage to the case of Jacobi forms is aided by the Fourier--Jacobi expansion. To get to the desired half--integral case cusp form, instead of using the Eichler--Zagier map (which is defined in \cite{manickam2000shimura}), we directly use the theta components of the Jacobi forms as in \cite{bocherer2018fundamental}. Although this passage to half--integral weight forms is inspired by the calculations in \cite{bocherer2018fundamental}, we are in a somewhat more technically challenging situation due to the existence of level and thus requires a fresh set of calculations. The use of theta components instead of the Eichler--Zagier map (as in \cite{saha2012determination}) lets us prove the result for all weights $k>2$, $k$ odd included.

By virtue of Corollary \ref{cor:primitive}, the problem reduces to half--integral weight cusp forms on the congruence subgroup $\Gamma_1(4L)$ for some $L\ge 1$. We then make use of the following result to finish the proof of Theorem \ref{thm:fund}.
\begin{thm}\label{thm:halfint}
Let $\kappa\ge 5/2$ be a half--integer and $L=\prod_{i=1}^{t}p_i^{\alpha_i}$. Suppose $f\in S_\kappa(\Gamma_1(4L))$ is non--zero and $a(f,n)=0$ for all $(n,L_f)>1$, for an  even divisor $L_f$ of $4L$. Then there exist infinitely many odd and square--free integers $n$ such that 
\begin{enumerate}
\item For $p_i|L$, $(n,p_i)=1$.
\item $a(f,p_1^{i_1}...p_r^{i_r}n)\neq 0$, where $(p_j,L_f)=1$ and  $0\le i_j\le \alpha_j$ for $1\le j\le r$.
\end{enumerate}
\end{thm}
On our way to prove Theorem \ref{thm:halfint}, we prove certain auxiliary results (e.g., Lemma \ref{lemma7analog}) for half--integral forms on $\Gamma_1(4L)$. Even though these results are known for the group $\Gamma_0(4L)$ from \cite{serre1977modular}, it seems such results are not readily available in the literature for the group $\Gamma_1(4L)$. So we give a brief account of them in section \ref{sec:halfint}.

\subsection*{Acknowledgments}
The author would like to thank S. Das for valuable comments and ideas that led to improved results. The author is a Postdoctoral Fellow at the Harish--Chandra Research Institute(HRI) and would like to thank HRI for funding and providing excellent facilities.
\section{Notations and preliminaries}
For any $z\in \mbb C$ and $m\in \mbb Z$, we write $e(z):= e^{2\pi i z}$, $e^m(z):= e^{2\pi i mz}$ and $e_m(z):= e^{2\pi i z/m}$.
\subsection{Jacobi forms with level}\label{sec:Jacobi} In this article we encounter Jacobi forms with level. For more details about Jacobi forms for congruence subgroups, see \cite{kramer1986jacobiformen}. 

A holomorphic function $\phi:\mbb H\times \mbb C\rightarrow \mbb C$ is a Jacobi form of weight $k$, index $m$ and level $N$ if:
\begin{enumerate}
\item $\phi\left(\frac{a \tau+b}{c \tau+d}, \frac{z}{c \tau+d}\right)(c \tau + d)^{-k} e^m\left(\frac{-c z^{2}}{c \tau+d}\right)=\phi(\tau,z)$\qquad for all $\gamma=\begin{psmallmatrix}
a&b\\
c&d
\end{psmallmatrix}\in \Gamma_0(N)$.
\item $\phi(\tau, z+\lambda\tau+\mu)e^m(\lambda^2\tau+2\lambda z)=\phi(\tau,z)$ \q\;\qquad for all $(\lambda,\mu)\in \mbb Z^2$.
\item $\phi$ satisfies the boundedness condition at the \textit{cusps}.
\end{enumerate}

We denote by $J_{k,m}(N)$ the space of Jacobi forms of weight $k$, index $m$ and level $N$. As a consequence of condition (3) above, any $\phi\in J_{k,m}(N)$ has a  Fourier expansion of the form (at the cusp $i\infty$)
\begin{equation}
\phi(\tau,w)=\sum\nolimits_{4nm-r^2\ge 0}c(n,r)e(n\tau)e(rw).
\end{equation}
A similar Fourier expansion holds at other cusps too. If the Fourier coefficients corresponding to $4nm-r^2=0$ vanish at all the cusps, then we say that $\phi$ is a Jacobi cusp form. We denote the space of Jacobi cusp forms by $J_{k,m}^{cusp}(N)$.

Like in the case of full level (\cite{eichler1985theory}), any $\phi\in J_{k,m}(N)$ admits a \textit{theta expansion}: 
\begin{equation}
\phi(\tau,w)=\sum\nolimits_{\mu\md 2m}h_\mu(\tau)\cdot \Theta_{\mu, m}(\tau,w),
\end{equation}
where $\Theta_{\mu, m}$ is the theta series given by
\begin{equation}
\Theta_{\mu, m}(\tau,w)=\sum_{n\in \mbb Z}e\Big(\frac{(2mn-\mu)^2}{4m}\tau+(2mn-\mu)w\Big)
\end{equation}
and $h_\mu$ are modular forms of weight $k-1/2$ for the congruence subgroup $\Gamma(4mN)$ (see \cite{bocherer2018fundamental} for a proof) and have a Fourier expansion as below.
\begin{equation}\label{Fhmu}
h_\mu(\tau)=\sum\nolimits_{4mn-\mu^2\ge 0}c(n,\mu)e\Big(\frac{4mn-\mu^2}{4m}\tau\Big).
\end{equation}
As in \cite{bocherer2018fundamental}, we refer to the theta components corresponding to $(\mu,2mN)=1$ as \textit{primitive} theta components.

For $0\le\mu< 2m$, $\Theta_{\mu, m}$ and $h_\mu$ satisfy the following transformation properties (see \cite{kramer1986jacobiformen}). Let $\gamma=\begin{psmallmatrix}
a&b\\c&d
\end{psmallmatrix}\in \mrm{SL}(2, \mbb Z)$, then
\begin{equation}\label{thetatrans}
\Theta_{\mu, m}\left(\frac{a \tau+b}{c \tau+d}, \frac{z}{c \tau+d}\right)(c \tau + d)^{-1 / 2} e\left(\frac{-mc z^{2}}{c \tau+d}\right)=\sum_{\nu\md 2m} \varepsilon_{m}(\nu, \mu;\gamma) \Theta_{\nu, m}(\tau, z).
\end{equation}
Here we take the branch of square root having argument in $(-\pi/2,\pi/2]$. 

From \eqref{thetatrans} and the transformation properties of $\phi$ under $\Gamma_0(N)$ (see conditions (1) and (2) above) we get,
\begin{equation}\label{hnurel}
(-N\tau+1)^{-k+1/2}h_\mu\Big(\frac{\tau}{-N\tau+1}\Big)=\sum_{\nu\md 2m} \varepsilon_{m}(\nu, \mu;\gamma)h_\nu(\tau),
\end{equation}
where $\gamma=\begin{psmallmatrix}
1&0\\N&1
\end{psmallmatrix}$ and $\varepsilon_{m}(\nu, \mu;\gamma)$ is the generalized quadratic Gauss sum and is given by
\begin{equation}\label{epsilonmunu}
\varepsilon_{m}(\nu, \mu;\gamma)=\frac{1}{2m}\sum_{\eta\md 2m}e_{4m}\left(N\eta^2+2\eta(\nu-\mu)\right).
\end{equation}
For the sake of simplicity we write $\varepsilon_{m}(\nu, \mu;\gamma)=\varepsilon_m(\nu, \mu)$. 

\subsection{Generalized quadratic Gauss sums}\label{sec:gauss} The generalized quadratic Gauss sum is defined as
\begin{equation}\label{gausssum}
G(a,b,c):=\sum\nolimits_{l\md c}e_c(a l^2+b l),\q \text{ where } a,b,c\in \mbb Z.
\end{equation}
$G(a,b,c)$ has the following properties (see \cite{berndt1998gauss} for more details):
\begin{equation}
G(a, b, c_1c_2)=G(c_2a,b,c_1)G(c_1a,b,c_2), \text{ when } (c_1,c_2)=1.
\end{equation}
Moreover, $G(a,b,c)=0$ if $(a,c)>1$ and $(a,c)$ does not divide $b$ and when $(a,c)|b$, we have
\begin{equation}
G(a,b,c)=(a,c)G\left(a(a,c)^{-1}, b(a,c)^{-1},c(a,c)^{-1}\right)
\end{equation}
and for $(a,c)=1$
\begin{equation}\label{gaussvalues}
G(a,b,c)=
\begin{cases}
\epsilon_c \sqrt{c}\big(\frac{a}{c}\big)e_c(-\psi(a)b^2)&\mbox{ if } c\equiv 1\md 2, 4a\psi(a)\equiv 1\md c;\\
2G(2a,b,c/2)&\mbox{ if } c\equiv 2\md 4,b\equiv 1\md 2;\\
0&\mbox{ if } c\equiv2\md 4, b=0 ;\\
(1+i)\epsilon_a ^{-1}\sqrt{c}\big(\frac{a}{c}\big)&\mbox{ if } c\equiv 0\md 4, b=0 ;\\
0&\mbox{ if } c\equiv 0\md 4, b\equiv 1\md 2.
\end{cases}
\end{equation}
Here $\epsilon_m= 1$ or $i$ depending on $m\equiv 1 \md 4$ or $3\md 4$ respectively.

Splitting the congruence class $\md 4m$ into two congruence classes $\md 2m$ we see from \eqref{epsilonmunu} and \eqref{gausssum} that $\varepsilon_m(\mu,\nu)=0$ when $(N,4m)\nmid (\nu-\mu)$ and when $(N,4m)| (\nu-\mu)$, we have
\begin{equation}\label{epgauss}
\varepsilon_m(\mu,\nu)=\frac{1}{4m}G(N, 2(\nu-\mu),4m)=\frac{(N,4m)}{4m}G\left(\frac{N}{(N,4m)}, \frac{2(\nu-\mu)}{(N,4m)},\frac{4m}{(N,4m)}\right).
\end{equation}

\subsection{Siegel modular forms}\label{sec:siegl} The Siegel's upper half space of degree $2$ is given by
\begin{equation}
\mathbb{H}_2=\{Z\in M(2,\mathbb{C})|Z=Z^t, \mrm{Im}(Z)>0\}.
\end{equation}
The symplectic group $\mrm{Sp}(2,\mbb{R})$ acts on $\mbb H_2$ by $Z\mapsto \gamma\langle Z\rangle = (AZ+B)(CZ+D)^{-1}$. For any integer $k$, the stroke operator on functions $F:\mbb H_2\rightarrow \mbb C$ is defined as
\begin{equation}
(F|_k\gamma)(Z):=\text{ det}(CZ+D)^{-k}F(\gamma\langle Z\rangle).
\end{equation}
For any positive integer $N$, the congruence subgroup $\Gamma_0^{2}(N)\subset \mrm{Sp}(2,\mbb Z)$ is given by
\begin{equation}
\Gamma_0^{2}(N)=\left\{\begin{psmallmatrix}
A&B\\C&D
\end{psmallmatrix}\in \mathrm{Sp}(2, \mathbb{Z})\;|\;C\equiv 0\md N\right\}.
\end{equation} 

A Siegel modular form of degree $2$, weight $k$ and level $N$ is a holomorphic function $F:\mbb H_2\rightarrow\mbb C$ satisfying $F|_k\gamma=F$ for all $\gamma\in \Gamma_0^{2}(N)$. Any such form admits a Fourier expansion of the form
\begin{equation}\label{Fourexp}
F(Z)=\sum\nolimits_{T\in \Lambda_2}A(F,T)e(\mrm{Tr}(TZ)),
\end{equation}
where $\Lambda_2$ is the set of $2\times 2$ half--integral, symmetric, positive semi--definite matrices. Denote the subset of positive definite matrices by $\Lambda^+_2$. If $A(F,T)$ non--zero for only $T\in \Lambda^+_2$ we say that $F$ is a cusp form. The space of Siegel modular forms of degree $2$, weight $k$ and level $N$ is denoted by $M_k^2(N)$ and the subspace of cusp forms is denoted by $S_k^2(N)$. 

For a $T=\begin{psmallmatrix}
a&b/2\\
b/2&c
\end{psmallmatrix}\in \Lambda_2$, define the content $c(T)$ as below:
\begin{equation}
c(T)=\max\{d\in \mbb N \;| \;d^{-1}T \text{ is half--integral}\}.
\end{equation}
When $c(T)=1$, we say that $T$ is \textit{primitive}. If $-D=b^2-4ac (<0)$ is a fundamental discriminant, then we say that $T$ is \textit{fundamental} and the corresponding Fourier coefficient $A(F,T)$ is called a fundamental Fourier coefficient. The same analogy applies to primitive Fourier coefficients.

\subsection{Fourier--Jacobi expansion} For any $Z\in \mbb H_2$ consider the following decomposition:
\begin{equation}
Z=\begin{psmallmatrix}
\tau&w\\
w&z
\end{psmallmatrix}; \text{ with } w\in \mbb C, \tau, z\in \mbb H \text{ (the complex upper half plane)}.
\end{equation}
Any $F\in M_k^2(N)$ has a Fourier--Jacobi expansion with respect to the above decomposition as below.
\begin{equation}\label{FJexp}
F(Z)=\sum\nolimits_{m\ge0}\phi_m(\tau,w)e(mz).
\end{equation}
Here $\phi_m$ are the Jacobi forms of weight $k$, index $m$ and level $N$ defined in section \ref{sec:Jacobi}. If $F$ is a cusp form then $\phi_m$ are also cusp forms.

For a $\phi\in J_{k,m}(N)$ that appears in the Fourier--Jacobi expansion of a Siegel modular form $F$ as in \eqref{FJexp}, the Fourier coefficients of $\phi$ are given by the Fourier coefficients of $F$:
\begin{equation}
c_\phi(n,r)=A\left(F, \begin{psmallmatrix}
n&r/2\\
r/2&m
\end{psmallmatrix}\right).
\end{equation}
\subsection{A newspace of Siegel cusp forms}\label{sec:newforms} In this article we deal with the Atkin-Lehner type space of newforms introduced in \cite{ibukiyama2012atkin}. We give a brief description of the space below. 

Let $N$ be any positive integer and $d (>1)$, $M$ be positive integers such that $dM|N$. Consider any $F\in S_k^2(M)$, then $F(dZ)\in S_k^2(N)$. Denote by $S_k^{2,old}(N)$ the subspace of $S_k^2(N)$ spanned by 
\begin{equation}
\{F(dZ)|F\in S_k^2(M), dM|N\}.
\end{equation}
The orthogonal complement of $S_k^{2,old}(N)$ with respect to the Petersson inner product is denoted by $S_k^{2,new}(N)$ and this is the subspace of $S_k^2(N)$ that we refer to as the Atkin--Lehner type \textit{newspace} in this article.
\section{Theta components of Jacobi forms with level}
Let $N$ be an odd integer and $m_1$, $m_2$ be square--free integers such that $m_1|N$ and $(N,m_2)=1$. For any $\phi\in J_{k,m_1m_2}(N)$, let $h_\mu$ denote the \textit{theta} components of $\phi$ (see section \ref{sec:Jacobi}). We also have that $h_\mu$ is a modular form of weight $k-1/2$ for the congruence subgroup $\Gamma(4m_1m_2N)$ and the Fourier expansion of $h_\mu$ in \eqref{Fhmu} can be rewritten as
\begin{equation}\label{hmuf}
h_\mu(\tau)=\sum\nolimits_{n}c_\phi\Big(\frac{n+\mu^2}{4m_1m_2},\mu\Big)q^{n/4m_1m_2},
\end{equation}
where we follow the convention that $c_\phi\Big(\frac{n+\mu^2}{4m_1m_2},\mu\Big)=0$ if $n\not\equiv-\mu^2\md 4m_1m_2$.

For any $\phi\in J_{k,m_1m_2}(N)$, suppose that $h_\mu=0$ for all $(\mu,2m_1m_2)=1$. Then from \eqref{hnurel} we get that for $(\mu,2m_1m_2)=1$:
\begin{equation}\label{sumhmu=0}
\underset{(\nu,2m_1m_2)>1} {\sum\nolimits_{\nu\md 2m_1m_2}}\varepsilon_{m_1m_2}(\nu, \mu)h_\nu(\tau)=0.
\end{equation}
First we evaluate the Gauss sums $\varepsilon_{m_1m_2}(\nu,\mu)$. Since $(N,m_2)=1$ and $m_1|N$, we have that $(N,4m_1m_2)=m_1$. Thus with $N=Mm_1$, we see from \eqref{epgauss} that (for $m_1|(\nu-\mu)$)
\begin{equation}\label{epsilongauss}
\varepsilon_{m_1m_2}(\nu, \mu)=\frac{1}{4m_2}G\left(M, \frac{2(\nu-\mu)}{m_1},4m_2\right).
\end{equation}
\begin{lem}
If $m_1|(\nu-\mu)$, then
\begin{equation}\label{epsilonmunu-gauss}
\varepsilon_{m_1m_2}(\nu, \mu)=\frac{1}{2\sqrt{m_2}}(1+i)\epsilon_{M}^{-1}\left(\frac{M}{4m_2}\right)e_{4m_2}\left(\frac{-\overline{M}(\nu-\mu)^2}{m_1^2}\right),
\end{equation}
where $\overline{M}\in \mbb Z$ such that $M\overline{M}\equiv 1\md 4m_2$.
\end{lem}
When $m_1|(\nu-\mu)$, the Gauss sum on the RHS can be evaluated using \eqref{gaussvalues} as shown below:
\begin{equation}
\begin{split}
G(M,\tfrac{2(\nu-\mu)}{m_1},4m_2)&=\sum_{\eta\md 4m_2}e_{4m_2}\Big(M\eta^2+\tfrac{2\eta(\nu-\mu)}{m_1}\Big)
=e_{4m_2}\Big(\tfrac{\overline{M}(\nu-\mu)^2}{m_1^2}\Big)G(M,0,4m_2)\\
&=(1+i)\epsilon_{M}^{-1} \sqrt{4m_2}\Big(\tfrac{M}{4m_2}\Big)e_{4m_2}\Big(\tfrac{-\overline{M}(\nu-\mu)^2}{m_1^2}\Big).
\end{split}
\end{equation}
\begin{lem}\label{lem:ep}
If $(\nu,m_1)>1$, then $\varepsilon_{m_1m_2}(\nu,\mu)=0$.
\end{lem}
\begin{proof}
First, note that the RHS of \eqref{epsilongauss} is zero whenever $m_1\nmid (\nu-\mu)$ (see section \ref{sec:gauss}). On the other hand, we also have that $(\mu,2m_1m_2)=1$ and $(\nu,2m_1m_2)>1$. Now if $(\nu,m_1)>1$, since $(\mu,m_1)=1$, we have $m_1\nmid (\nu-\mu)$. Thus $\varepsilon_{m_1m_2}(\nu,\mu)=0$ in this case.
\end{proof}

Now coming back to \eqref{sumhmu=0}, using the transformation $\tau\mapsto\tau+t$ we get
\begin{equation}
\underset{(\nu,2m_1m_2)>1} {\sum\nolimits_{\nu\md 2m_1m_2}}\varepsilon_{m_1m_2}(\nu, \mu)e_{4m_1m_2}(-\nu^2t)h_\nu(\tau)=0.
\end{equation}
Using the linear independence of pairwise different characters (see \cite[pp. 15] {bocherer2018fundamental} for a detailed argument) we see that it is enough to consider $\nu$ with $(\nu,2m_1m_2)>1$ in square classes. That is, $\nu^2\equiv \nu_0^2 \mod 4m_1m_2$ for fixed $\nu_0 \mod 2m_1m_2$. Thus from \eqref{epsilonmunu} and Lemma \ref{lem:ep} we get:
\begin{equation}\label{hnucomb}
\underset{(\nu_0,m_1)=1} {\underset{\nu^2\equiv\nu_0^2\md 4m_1m_2}{\sum_{\nu\md 2m_1m_2}}}e_{4m_2}\Big(\frac{-\overline{M}(\nu-\mu)^2}{m_1^2}\Big)h_\nu(\tau)=0.
\end{equation}
For $(\nu_0,m_1)=1$ and any $l$ with $(l,2m_2)=1$, consider the matrix
\begin{equation}
M_{\nu_0}=\left(\varepsilon_{m_1m_2}(\nu,\mu)\right)_{\nu,\mu},
\end{equation}
where $\varepsilon_{m_1m_2}(\nu,\mu)=e_{4m_2}\left(\frac{l(\nu-\mu)^2}{m_1^2}\right)$ when $m_1|(\nu-\mu)$ and is zero otherwise, $(\mu,2m_1m_2)=1$ and $\nu$ runs over the set
\begin{equation}\label{setS}
S_{\nu_0}=\{\nu\md 2m_1m_2:  \nu^2\equiv\nu_0^2\md 4m_1m_2\}.
\end{equation}
We prove that the only possible solution to \eqref{hnucomb} is the trivial solution (i.e., $h_\nu=0$) by showing that the matrix $M_{\nu_0}$ has maximal rank. When $N=1$, this reduces to the arguments presented in \cite{bocherer2018fundamental}.
\begin{lem}
Let $(\nu_0,m_1)=1$. Then the matrix $M_{\nu_0}$ has maximal rank.
\end{lem}
\begin{proof}
Let $m_1=p_1p_2...p_r$ and $m_2=q_1q_2...q_s$. Write $t=r+s$. The set
\begin{equation}
\{\nu\md 2m_1m_2: \nu^2\equiv\nu_0^2\md 4m_1m_2\}
\end{equation}
has the cardinality $2^{t'}$, where $t'$ is the number of primes dividing $m_1m_2$ but not $\nu_0$. But the additional condition $(\nu_0,m_1)=1$ gives us that the cardinality of the above set is at least $2^r$. We prove by induction on $t$.

When $t=0$, $|S_{\nu_0}|=1$ and the lemma is trivially true. Consider the case $t=1$. Then either $r=1$ and $s=0$ or $r=0$ and $s=1$.
\begin{itemize}
\item[Case 1:] $r=1$ and $s=0$. In this case $|S_{\nu_0}|=2$ with $\nu_0$ and $-\nu_0$ being the solutions. Choose $\mu_1=\nu_0+2p_1$ ($\nu_0+p_1$ if $\nu_0$ is even) and $\mu_2=-\nu_0+2p_1$ ($-\nu_0+p_1$ if $\nu_0$ is even). Then clearly $(\mu_i,2p_1)=1$ for $i=1,2$. Clearly the diagonal matrix
\begin{equation}
\begin{pmatrix}
e_{4}\big(\frac{l(\nu_0-\mu_1)^2}{p_1^2}\big)& 0\\
0&e_{4}\big(\frac{l(-\nu_0-\mu_2)^2}{p_1^2}\big)
\end{pmatrix}
\end{equation}
has non--zero determinant. Thus $M_{\nu_0}$ has rank $2$.
\item[Case 2:]$r=0$ and $s=1$. If $q_1|\nu_0$, then $|S_{\nu_0}|=1$  and the matrix $M_{\nu_0}$ is a non--zero column and thus the lemma follows. If $(\nu_0,q_1)=1$, then $|S_{\nu_0}|=2$. Determinant of the matrix 
\begin{equation}
\begin{pmatrix}
e_{4q_1}\big(l(\nu-1)^2\big)& e_{4q_1}\big(l(\nu+1)^2\big)\\
e_{4q_1}\big(l(\nu-\mu)^2\big)&e_{4q_1}\big(l(\nu+\mu)^2\big)
\end{pmatrix}
\end{equation}
is given by $e_{4q_1}\left(l(2\nu^2+\mu^2+2\nu(\mu-1))\right)\left(1-e_{q_1}\left(l\nu(\mu-1)\right)\right)$. Choosing $\mu\md 2q_1$ with $(\mu,2q_1)=1$ and different from $1$, we see that this determinant is non-zero. Thus the matrix $M_{\nu_0}$ has rank $2$.
\end{itemize}
Now coming to the induction step $t\implies t+1$, first we assume that $r>0$ and write $m_1=m_3p$ and decompose $\nu_0\md 2m_1m_2$ as $\nu_0=2m_3m_2\nu_0'+p\nu_0''$ with $\nu_0'\md p$ and $\nu_0''\md 2m_3m_2$. Decompose $\nu$ and $\mu$ similarly. Note that 
\begin{enumerate}
\item $m_1|(\nu-\mu)$ iff $m_3|(\nu''-\mu'')$ and $p|(\nu'-\mu')$.
\item $\nu^2\equiv \nu_0^2\md 4m_1m_2$ iff $\nu'^2\equiv\nu_0'^2\md p$ and $\nu''^2\equiv\nu_0''^2\md 4m_3m_2$.
\end{enumerate}
Moreover, we also have
\begin{equation}
\frac{(\nu-\mu)^2}{4m_1^2m_2}=\frac{m_2(\nu'-\mu')^2}{p^2}+\frac{(\nu''-\mu'')^2}{4m_3^2m_2}\md \mbb Z.
\end{equation}
As $\nu'$ and $\nu''$ run over the sets $\{\nu'\md p:  \nu'^2\equiv\nu_0'^2\md p\}$ and $\{\nu''\md 2m_3m_2:  \nu''^2\equiv\nu_0''^2\md 4m_3m_2\}$ respectively, write $\varepsilon'(\nu',\mu')=e\left(\frac{lm_2(\nu'-\mu')^2}{p^2}\right)=1$ if $p|(\nu'-\mu')$ and is zero otherwise and $\varepsilon''(\nu'',\mu'')=e_{4m_2}\left(\frac{l(\nu''-\mu'')^2}{m_3^2}\right)$ if $m_3|(\nu''-\mu'')$ and zero otherwise. Let
\begin{equation}
A:=\left(\varepsilon'(\nu',\mu')\right)_{\nu',\mu'\md p}\q\text{ and }\q B:=\left(\varepsilon''(\nu'',\mu'')\right)_{\nu'',\mu''\md 2m_3m_2},
\end{equation}
From the induction hypothesis, $B$ has maximal rank. To see that $A$ also has maximal rank, first note that $(\nu_0',p)=1$ since $(\nu_0,m_1)=1$. Thus $\nu'^2\equiv \nu_0'^2\md p$ has two solutions (i.e., the rank of $A$ is at most $2$). Denote them by $\nu_1'=\nu_0'$ and $\nu_2'=-\nu_0$. Let $\mu_1'=\nu_1'+p$ and $\mu_2'=\nu_2'+p$. Clearly $(\mu_i',p)=1$ and $p|(\nu_i'-\mu_j')$ iff $i=j$. Thus $A$ has a sub-matrix of the form $\begin{psmallmatrix}
1&0\\0&1
\end{psmallmatrix}$. 
This shows that $A$ is indeed of rank $2$. The lemma is true in this case by noting that $M_{\nu_0}=A\otimes B$.

Now we turn to the case when $r=0$. Here $s=t+1$ and we write $m_2=m_4q$, where $m_4$ has $t=s-1$ prime factors. Note that in this case $m_1=1$. As in the previous case decompose $\nu_0\md 2m_2$ as $\nu_0=2m_4\nu_0'+q\nu_0''$ with $\nu_0'\md q$ and $\nu_0''\md 2m_4$. Decompose $\nu$ and $\mu$ similarly. Note that $\nu^2\equiv \nu_0^2\md 4m_2$ iff $\nu'^2\equiv\nu_0'^2\md q$ and $\nu''^2\equiv\nu_0''^2\md 4m_4$. We also have
\begin{equation}
\frac{(\nu-\mu)^2}{4m_2}=\frac{m_4(\nu'-\mu')^2}{q}+\frac{q(\nu''-\mu'')^2}{4m_4}\md \mbb Z.
\end{equation}
As before, we write 
\begin{equation}
A:=\left(\varepsilon'(\nu',\mu')\right)_{\nu',\mu'\md q}\q\text{ and }\q B:=\left(\varepsilon''(\nu'',\mu'')\right)_{\nu'',\mu''\md 2m_4},
\end{equation}
where $\varepsilon'(\nu',\mu')=e_q\left(lm_4(\nu'-\mu')^2\right)$ and $\varepsilon''(\nu'',\mu'')=e_{4m_4}\left(lq(\nu''-\mu'')^2\right)$. As in the previous case, $\nu'$ runs over the set $\{\nu'\md q:  \nu'^2\equiv\nu_0'^2\md q\}$ and $\nu''$ runs over the set $\{\nu''\md 2m_4:  \nu''^2\equiv\nu_0''^2\md 4m_4\}$.  

Since $(l,2m_4)=1$, the matrix $B$ has maximal rank by induction hypothesis. To see that $A$ has maximal rank, note that $\nu'^2\equiv\nu_0'^2\md q$ has one solution if $q|\nu_0'$ and in this case $A$ is a non--zero column and thus this is true.

If $(\nu_0',q)=1$, then $\nu'^2\equiv\nu_0'^2\md q$ has two solutions. The matrix 
\begin{equation}
\begin{pmatrix}
e_{q}\big(lm_4(\nu'-1)^2\big)& e_{q}\big(lm_4(\nu'+1)^2\big)\\
e_{q}\big(lm_4(\nu'-\mu)^2\big)&e_{q}\big(lm_4(\nu'+\mu)^2\big)
\end{pmatrix}
\end{equation}
has determinant $e_{q}\left(lm_4(2\nu'^2+\mu'^2+1+2\nu'(\mu'-1))\right)\left(1-e_{q}\left(4lm_4\nu'(\mu'-1)\right)\right)$. By choosing $\mu$ mod $q$ with $(\mu,q)=1$ and different from $1$, we see that the above determinant is non-zero. Thus the matrix $A$ has rank $2$. By noting that $M_{\nu_0}=A\otimes B$, the lemma is now complete.
\end{proof}
Now define the relation $\nu\sim\nu'$ iff $\nu^2\equiv\nu'^2\md 4m_1m_2$. This defines an equivalence relation among $\nu\md 2m_1m_2$ with $(\nu,m_1)=1$ and $(\nu,2m_1m_2)|2m_2$. Thus from the above discussion we see that $h_\nu=0$ for all such indices. This gives us the proof of Theorem \ref{th:thetanz}.

When $m_1=1$, we have the following corollary.
\begin{cor}\label{cor:primitive}
Let $N$ be an odd integer and $m$ be a square--free integer with $(N,m)=1$. For any non--zero $\phi\in J_{k,m}(N)$ there exists a $\mu \md 2m$ with $(\mu,2m)=1$ such that the theta component $h_\mu\neq 0$.
\end{cor}

\section{Half--integral weight forms for \texorpdfstring{$\Gamma_1(4L)$}{Gamma1L}}\label{sec:halfint}
For any positive integer $L$ and a positive half--integer $\kappa$ (i.e., $\kappa=l/2$ for some odd $l>0$) denote by $S_\kappa(\Gamma_1(4L))$ the space of half--integral weight cusp forms for the congruence subgroup $\Gamma_1(4L)$. Let $\chi$ be any Dirichlet character modulo $4L$, then the space $S_\kappa(\Gamma_1(4L))$ decomposes as below.
\begin{equation}\label{directdecomp}
S_\kappa(\Gamma_1(4L))=\oplus_{\chi\md 4L} S_\kappa(4L,\chi),
\end{equation}
where $S_\kappa(4L,\chi)$ denotes the space of half--integral weight cusp forms for the congruence subgroup $\Gamma_0(4L)$ and character $\chi$.

Let $L$ be as above and $L_f$ an even divisor of $4L$. Consider a non--zero $f\in S_\kappa(\Gamma_1(4L))$  such that $a(f,n)=0$ for all $(n,L_f)>1$. Write 
\begin{equation}\label{Ldecomp}
L=\prod\nolimits_{i=1}^{t}p_i^{\alpha_i}\;,\; L':=\prod\nolimits_{i=1}^{t}p_i\;,\;L_f'=\prod\nolimits_{p|L_f}p\;,\; L'=M_fL_f'.
\end{equation}

Our aim is to construct a new modular form $g$ such that $a(g,n)=0$ for all $(n,L')>1$. To proceed further, we need the following two results. The first one is an analogue of Lemma 7 in \cite{serre1977modular} for the congruence subgroup $\Gamma_1(4L)$. The arguments used here are similar in spirit to those used to prove a similar result in the integral weight case (see \cite{lang2012introduction}) with some additional care.
\begin{lem}\label{lemma7analog}
Let $h\in S_\kappa(\Gamma_1(4L))$ be non--zero. For an odd prime $p|L$, suppose $a(h,n)=0$ for all $(n,p)=1$, then $g(\tau):=h(\tau/p)$ belongs to $S_\kappa(\Gamma_1(4L/p))$.
\end{lem}
\begin{proof}
Consider the operators $V_p$ and $U_p$ given as below (see \cite{serre1977modular} for more about these operators).
\begin{equation*}
\begin{split}
&(f|V_p)(\tau)=f(p\tau)=p^{-\kappa/4}f|\left\{\begin{psmallmatrix}
p&0\\
0&1
\end{psmallmatrix}, p^{-1/4}\right\}=\sum a(f,n)q^{np},\\
&(f|U_p)(\tau)=p^{\kappa/4-1}\sum\nolimits_{j\text{ mod } p}f|\left\{\begin{psmallmatrix}
1&j\\
0&p
\end{psmallmatrix}, p^{1/4}\right\}=\sum a(f,np)q^{n}.
\end{split}
\end{equation*}
Then the condition on $h$ is equivalent to $h|(1-U_pV_p)=0$.

Next we consider the projection $\pi_\chi:S_\kappa(\Gamma_1(4L))\rightarrow S_\kappa(4L,\chi)$ which is given by 
\begin{equation}
\pi_\chi(f)=\frac{1}{\phi(4L)}\sum_{a\md 4L, (a,4L)=1}\overline{\chi(a)}f|\sigma_a^*,
\end{equation}
$\sigma_a\in \mrm{SL}_2(\mbb Z)$ such that $\sigma_a\equiv \begin{psmallmatrix}
a^{-1}&0\\
0&a
\end{psmallmatrix}\md 4L$. Also note that if $\sigma_a$ and $\sigma_a'$ are any two such matrices, then $\sigma_a\sigma_a'^{-1}\equiv \begin{psmallmatrix}
1&0\\0&1
\end{psmallmatrix}\md 4L$ and thus we have that $f|\sigma_a^*=f|\sigma_a'^*$. If we write $h=\sum h_\chi$, then we have that $h_\chi=\pi_\chi(h)$.

Now we claim that $h|(1-U_pV_p)=0$ would imply $h_\chi|(1-U_pV_p)=0$ for every $\chi\md 4L$. To see this, we show that actions of $U_pV_p$ and $\sigma_a^*$ commute. Let $\sigma_a=\begin{psmallmatrix}
a_1&b\\
c&d
\end{psmallmatrix}$ with $a_1\equiv a^{-1}\md 4L$,  $b\equiv c\equiv 0\md 4L$ and $d\equiv a\md 4L$. First note that
\begin{equation}
\left\{\begin{psmallmatrix}
p&j\\
0&p
\end{psmallmatrix},1\right\}\sigma_a^*=\sigma_a'^*\left\{\begin{psmallmatrix}
p&j'\\
0&p
\end{psmallmatrix},1\right\},
\end{equation}
where $\sigma_a'=\begin{psmallmatrix}
a_1+jc&(b+jd)-j'(a_1+jc)\\
c&d-j'c
\end{psmallmatrix}$. By choosing $j'\equiv a_1^{-1}dj\md 4L$, we have $f|(U_pV_p)|\sigma_a^*=f|\sigma_a'^*|(U_pV_p)$. Thus
\begin{equation}
h_\chi|(1-U_pV_p)=\frac{1}{\phi(4L)}\sum_{a\md 4L, (a,4L)=1}\overline{\chi(a)}h|(1-U_pV_p)|\sigma_a^*=0.
\end{equation}
Thus we get that for every $\chi\md 4L$, $h_\chi\in S_\kappa(4L,\chi)$ such that $a(h_\chi,n)=0$ for all $(n,p)=1$.

Now $g(\tau)=h(\tau/p)=\sum h_\chi(\tau/p)$. If $h_\chi$ is non--zero, then from \cite[Lemma 7]{serre1977modular}, $h_\chi(\tau/p)\in S_\kappa(4L/p,\chi\chi_p)$. Thus we see that $g\in S_\kappa(\Gamma_1(4L/p))$.
\end{proof}
The second result is the following lemma from \cite[Lemma 4.2]{bocherer2018fundamental}.
\begin{lem}\label{lemsbsd}
Let $h\in S_\kappa(\Gamma_1(4L))$. Suppose $a(h,n)=0$ for all $(n,p)=1$ for an odd prime $p\nmid L$, then $h=0$.
\end{lem}

We now construct the new cusp form $g$ using Lemma \ref{lemma7analog} and \ref{lemsbsd} and an inductive argument. These arguments are similar in essence to the ones used in \cite{saha2013siegel} for the congruence subgroup $\Gamma_0(4N)$ and in \cite{anamby2019distinguishing} for the integral weight cusp forms.

Let $M_f,L_f'$ be the co--prime square--free integers as in \eqref{Ldecomp}. Write $M_f=p_1p_2....p_r$. 
\begin{enumerate}
\item Consider
\begin{equation*}
g_0(\tau)=\sum\nolimits_{(n,L_f')=1}a(f,n)q^n.
\end{equation*}
Clearly $g_0\neq 0$ and $g_0\in S_\kappa(\Gamma_1(4LL_f'^2))$. We construct $g_1$ such that $a(g_1,n)=0$ for $(n,p_1L_f')>1$. 
\begin{enumerate}
\item  Let $g_{1,0}(\tau)=\sum\nolimits_{(n,p_1)=1}a(g_0,n)q^n.$  Then $g_{1,0}\in S_\kappa(\Gamma_1(4LL_f'^2p_1^2)$.
\item  If $g_{1,0}=0$, then we set $g_{1,0}'(\tau)=g_0(\tau/p_1)$. Clearly $g_{1,0}'\neq 0$, since $g_0\neq 0$. Since $g_{1,0}=0$, we have from Lemma \ref{lemma7analog} that $g_{1,0}'\in S_\kappa(\Gamma_1(4LL_f'^2/p_1))$. Now we set $g_{1,1}(\tau)=\sum\nolimits_{(n,p_1)=1}a(g_{1,0}',n)q^n$ and we have $g_{1,1}\in S_\kappa(\Gamma_1(4LL_f'^2p_1)$.
\item  Now suppose $g_{1,i}\in S_\kappa(\Gamma_1(4LL_f'^2p_1^{2-i})$ has been constructed as in step (b) for some $0\le i\le \alpha_1$. If $g_{1,i}=0$, then we go back to step (b) and construct $g_{1,i+1}\in S_\kappa(\Gamma_1(4LL_f'^2p_1^{2-(i+1)})$ with $g_{1,0}$ and $g_0$ replaced by $g_{1,i}$ and $g'{1,i-1}$ respectively. The first $i_0$ for which $g_{1,i_0}\neq 0$, we set $g_1=g_{1,i_0}$.
\end{enumerate}
If all of $g_{1,0},\cdots, g_{1,\alpha_1}$ are zero, since $g'_{1,\alpha_1-1}\in S_\kappa(\Gamma_1(4LL_f'^2/p_1^{\alpha_1}))$ and $p_1^{\alpha_1}||L$, we see from the Fourier coefficients of $g_{1,\alpha_1}$ and Lemma \ref{lemsbsd} that $g'_{1,\alpha_1-1}=0$. This gives us that $g_0=0$, a clear contradiction. Hence, the process must stop for some $1\le i\le \alpha_1$. Thus we get $g_1\in S_\kappa(\Gamma_1(4LL_f'^2p_1^{2-i_0})$ for some $1\le i_0\le \alpha_1$ such that $a(g_1,n)=0$ for $(n,p_1L_f')>1$ and $a(g_1,n)=a(g_0,p^{i_0}n)$.
\item Now suppose $g_j\in S_\kappa(\Gamma_1(4LL_f'^2p_1^{2-i_1}p_2^{2-i_2}...p_j^{2-i_j}))$ has been constructed for some $1\le j< r$, then we construct $g_{j+1}$ as in step (1) with $g_0$, $L_f'$ and $M_f$ replaced by $g_j$, $L_f'p_1p_2...p_j$ and $p_{j+1}...p_r$ respectively.

\item Finally, we set $g=g_r$ and we have $g\in S_\kappa(\Gamma_1(4LL_f'^2p_1^{2-i_1}p_2^{2-i_2}...p_r^{2-i_r}))$. Moreover, $a(g,n)=0$ if $(n,L_f'p_1...p_r)>1$ and $a(g,n)=a(f,p_1^{i_1}...p_r^{i_r}n)$ otherwise.
\end{enumerate}
We summarize the discussion in the following proposition.
\begin{prop}\label{rednewg}
Let $L=\prod_{i=1}^{t}p_i^{\alpha_i}$ and $L_f$ be an even divisor of $4L$. Suppose $f\in S_\kappa(\Gamma_1(4L))$ is non--zero and $a(f,n)=0$ for all $(n,L_f)>1$. Then there exists a $g\in S_\kappa(\Gamma_1(4LL_f'^2p_1^{2-i_1}...p_r^{2-i_r}))$ with the following properties.
\begin{enumerate}
\item $a(g,n)=0$ whenever $(n,4L)>1$.
\item $a(g,n)=a(f,p_1^{i_1}...p_r^{i_r}n)$ when $(n,4L)=1$.
\item $0\le i_j\le \alpha_j$ for $1\le j\le r$.
\end{enumerate}
\end{prop}
Now we use the following result for half--integral weight cusp forms from \cite[Theorem 4.5]{bocherer2018fundamental}.
\begin{thma}\label{sqfsbsd}
Let $\kappa\ge 5/2$ be a half--integer and $L\ge1$ be an integer. Suppose $g\in S_\kappa(\Gamma_1(4L))$ is such that $a(g,n)=0$ whenever $(n,4L)>1$. Then there exists infinitely many odd and square-free integers $n$ such that $a(g,n)\neq 0$. More precisely, for any $\epsilon>0$,
$\#\{n\le X, n \text{ square--free }:\; a(g,n)\neq 0\}\gg_{g,\epsilon}X^{5/8-\epsilon}.$
\end{thma}
To complete the proof of Theorem \ref{thm:halfint}, first note that for any such $f$ as in the statement of Theorem \ref{thm:halfint} we get from Proposition \ref{rednewg} a $g\in S_\kappa(\Gamma_1(4L_1))$ whose Fourier coefficients are supported away from the level. Now using Theorem \ref{sqfsbsd} for $g$ and noting the relation between the Fourier coefficients of $f$ and $g$ we get Theorem \ref{thm:halfint}.\qed

\section{Fourier coefficients of Siegel cusp forms}
Using the corresponding results for Jacobi forms (Theorem \ref{th:thetanz}) and half integral weight forms (Theorem \ref{thm:halfint}) we now give the proof of Theorem \ref{thm:fund}.
First we need the following result due to Ibukiyama and Katsurada (see \cite{ibukiyama2012atkin}).
\begin{thma}
Let $F\in S_k^{2,new}(N,\chi)$ be non--zero. Then there exists a primitive $T$ such that $A(F,T)\neq 0$. 
\end{thma}
 For any $M\in \mrm{GL}(2, \mbb Z)$, we have that $A(F, T)=\mrm{det}(M)^kA(F, M^tTM)$. Thus if $A(F,T)\neq 0$, we can say that $A(F, M^tTM)\neq 0$. Now coming to the case at hand,  the quadratic form associated to $T$ represents infinitely many primes (see \cite{weber1882beweis}). Since $T$ is primitive, we can choose $M$ such that the matrix $T_0=M^tTM$ has the right lower entry to be an odd prime, say $p$ (see \cite[Lemma 2.1]{saha2013siegel} for exact arguments). Since the quadratic form defined by $T$ represents infinitely many odd primes, we can choose $p$ such that $(p,N)=1$. Let $T_0=\begin{psmallmatrix}
n_0&\mu_0/2\\
\mu_0/2& p
\end{psmallmatrix}$ and we have $A(F,T_0)\neq 0$. Thus we get a non--zero Fourier--Jacobi coefficient $\phi_{p}\in J_{k,p}(N)$ of $F$. Let $\phi_{p}(\tau,z)=\sum\nolimits_{ 4pn>r^2}c(n,r)e(n\tau)e(rz),
$
where $c(n,r)=A\left(F,\begin{psmallmatrix}
n&r/2\\
r/2& p
\end{psmallmatrix}\right)$. Since $A(F, T_0)\neq 0$, we see that $c(n_0,\mu_0)\neq 0$.
\begin{prop}\label{jacobiprime}
Let $F \in S_k^{2,new}(N)$ be non--zero. Then there exists an odd prime $p$ with $(p,N)=1$ such that the Fourier--Jacobi coefficient $\phi_{p}\in J_{k,p}^{cusp}(N)$ of $F$ is non--zero.
\end{prop}
Write $N=p_1^{\alpha_1}p_2^{\alpha_2}...p_t^{\alpha_t}$. Then from Proposition \ref{jacobiprime} and Corollary \ref{cor:primitive} we get a $\mu\md 2p$ with $(\mu,2p)=1$ such that $h_\mu\neq 0$. Let $f(\tau)=h_\mu(4p\tau)$, then $f\in S_{k-\frac{1}{2}}(\Gamma_1((4p)^2N))$. From the Fourier expansion of $h_\mu$ in \eqref{hmuf}, we see that $f(\tau)=\sum\nolimits_{n\ge 1} a(f, n)q^n,$ where $a(f, n)=c\Big((n+\mu^2)/4p,\mu\Big)$.

Now using the Theorem \ref{thm:halfint} for level $4L=16p^2N$, we get primes $p_i|N$ and $0\le r_i\le \alpha_i$ and infinitely many odd, square--free integers $n$ with $(n,2pN)=1$ such that $a(f,p_1^{r_1}p_2^{r_2}...p_t^{r_t}n)\neq 0$. Thus for infinitely many odd and square--free integers $n$ co-prime to $N$ we see that
\begin{equation*}
a(f,p_1^{r_1}p_2^{r_2}...p_t^{r_t}n)=c\Big(\frac{p_1^{r_1}p_2^{r_2}...p_t^{r_t}n+\mu^2}{4p},\mu\Big)=A\left(F,\begin{psmallmatrix}
\frac{p_1^{r_1}p_2^{r_2}...p_t^{r_t}n+\mu^2}{4p} &\mu/2\\
\mu/2&p
\end{psmallmatrix}\right)\neq 0.
\end{equation*}
 If we denote by $T=\begin{psmallmatrix}
\frac{p_1^{r_1}p_2^{r_2}...p_t^{r_t}n+\mu^2}{4p} &\mu/2\\
\mu/2&p
\end{psmallmatrix}$, then $4\text{ det}(T)=p_1^{r_1}p_2^{r_2}...p_t^{r_t}n$. Thus we get that $A(F, T)\neq 0$ for infinitely many $T$ such that $4\text{ det}(T)$ is of the form $p_1^{r_1}p_2^{r_2}...p_t^{r_t}n$, where $n$ is odd and square--free and co-prime to $N$. \qed

\begin{rmk}
For any $X\ge 1$, let $\mathcal S_F(X)$ denote the set of odd and square--free integers $n\le X$ such that $(n,N)=1$, $\prod_{1}^{t}p_j^{r_j} n=4\mrm{det}(T)$ for some $T\in \Lambda_2^+$ and $A(F,T)\neq 0$.
The quantitative result $\mathcal S_F(X)\gg_{F,\epsilon} X^{5/8-\epsilon}$ follows from the corresponding quantitative result for half--integral weight cusp forms.
\end{rmk}

\begin{rmk}
The condition $k>2$ on the weight comes from the corresponding condition on the half--integral weight modular forms. For a discussion on this see \cite[Remark 4.7]{bocherer2018fundamental}.
\end{rmk}

\begin{rmk}
To get the \textit{fundamental} Fourier coefficients in Theorem \ref{thm:fund}, we need a non--zero \textit{primitive} theta component in Corollary \ref{cor:primitive}. That is, $h_\mu\neq 0$ with $(\mu, 2mN)=1$. But in Corollary \ref{cor:primitive} we get a theta component $h_\mu\neq 0$ such that $(\mu,2m)=1$. This is because the method used in proving Theorem \ref{th:thetanz} \textit{does not see} the level $N$. It will be interesting to see if Theorem \ref{th:thetanz} can be improved to get a non--zero primitive theta component.
\end{rmk}

\begin{rmk}
Theorem \ref{thm:fund} is not true  in $S_k^{2,old}(N)$. As an example, let $N$ be square--free and let $d$ be a proper divisor of $N$. Then the Fourier coefficients of $F(dZ)$, where $F\in S_k^{2,new}(N/d)$, are supported on $T\in \Lambda_2^+$ for which $4\text{ det}(T)$ is of the form $d^2 n$.
\end{rmk}

\printbibliography
\end{document}